\documentclass{article}
\usepackage{amsfonts,amsmath}
        \newtheorem{theorem}{Theorem}
 
        \newtheorem{example}[theorem]{Example}
        \newtheorem{lemma}[theorem]{Lemma}
        \newtheorem{corollary}[theorem]{Corollary}
        \newtheorem{remark}[theorem]{Remark}

\usepackage{esvect}
\usepackage{epsfig}
\usepackage{xcolor}
\title{On condition numbers of symmetric and nonsymmetric domain decomposition methods}

\author{Juan Galvis\thanks{Departamento de Matem\'aticas, Universidad Nacional de Colombia, Bogot\'a, Colombia.}}

\begin{document}
\maketitle
\tableofcontents

\begin{abstract}
    Using  oblique projections and angles between subspaces 
    we write condition number estimates for abstract nonsymmetric domain decomposition methods. In particular,  we 
    consider a restricted additive method for the Poisson equation and write a bound for the condition number of the preconditioned operator. We also obtain the non-negativity of the preconditioned operator. Condition number estimates are not enough for the convergence of iterative methods such as GMRES but these bounds may lead to further understanding of nonsymmetric domain decomposition methods. 
    \\
    
{\bf Keywords:}
Restricted Additive Schwarz, Domain Decomposition Methods, Oblique Projections.
\end{abstract}

\section{Introduction}
The restricted additive Schwarz (RAS) was originally introduced by 
Cai and Sarkis in \cite{CaiSarkis} in 1999.  
 RAS outperforms the classical additive Schwarz (AS) preconditioner
 in the sense that it requires fewer iterations, as well as lower
communication and CPU time costs when implemented on distributed memory computers 
\cite{CaiSarkis}. Unfortunately, RAS in its original form is nonsymmetric, and therefore the
conjugate gradient (CG) method cannot be used. Pursuing the analysis of RAS, several interesting 
methods have been developed. Some of these versions have been completely or partially 
analyzed and some of them outperform the classical AS. 
Despite of many contributions, the analysis of this 
method remains incomplete.  

We mention some of the developments related to the RAS method. 
The methods was introduced in \cite{CaiSarkis}.
The authors introduced the RAS as a cheaper and faster variants of the classical AS preconditioner for general sparse linear systems. The new method was shown to 
perform better that the AS according to 
the numerical studies presented there (see also \cite{CaiFarhatSarkis}). The authors 
of \cite{CaiSarkis} quoted that
\begin{quote}
...RAS was found accidentally. While working on a AS/GMRES algorithm in a Euler simulation, we
removed part of the communication routine and surprisingly the “then AS” method converged faster both
in terms of iteration counts and CPU time. We note that RAS is the default parallel preconditioner for
nonsymmetric sparse linear systems in PETSc ...
\end{quote}

Many works have been devoted to RAS and therefore it would be difficult to present a complete review of them. 
Here we mention that in \cite{Szyld,Szyld2} an algebraic convergence analysis is presented. 
In \cite{rasho,CaiFarhatSarkis} 
the authors provide and extension of RAS  using the
so-called harmonic overlaps (RASHO). Both RAS and RASHO outperform their
counterparts of the classical additive Schwarz variants.
An almost optimal convergence theory is
presented for the RASHO.  In \cite{Gander}, 
it is shown 
that a matrix interpretation of RAS  iteration can be related to the the continuous level of the 
underlying problem. The authors explain how this  interpretation reveals why RAS converges faster than classical AS. Still,  bounds for of the condition number of the 
RAS  preconditioned operator remains to be satisfactory. In \cite{MR2104179}, a by now classical book introducing domain decomposition methods, the authors comment
\begin{quote}
To our knowledge, a comprehensive theory of this algorithm is still missing. We note however that the restricted additive Schwarz preconditioner is the default parallel preconditioner for nonsymmetric systems in the PETSc library ...and has been used for the solution of very large problems...  
\end{quote}

 In this paper we re-visit the classical one-level additive method and the restricted additive method  proposed by Cai and Sarkis.   Inspired by these methods, we develop an abstract 
setting that may be useful for further understanding of nonsymmetric methods.  We write a Hilbert space framework for the analysis of the classical additive method. Then we generalize this Hilbert space framework and apply this extension to  write bounds for the condition numbers of several preconditioned operators where the construction of the preconditioner  uses restrictions onto original subdomains (instead of restrictions to the overlapping subdomains). We present abstract results that may be useful to analyze non-symmetric domain decomposition method in general. We illustrate in particular how to use the results for a one-level restricted additive method with similar local problems as the local problems of the OBDD (Overlapping Balancing Domain Decomposition) introduced in \cite{kimn2007obdd}. Several other models and similar methods can be considered as well. For instance, restricted method 
for the elasticity equation, two-level domain decomposition method with classical or modern coarse spaces design, etc.

The rest of the paper is organized as follows. 
In Section \ref{hilbertspace} we review the classical 
domain decomposition methods results in a simple Hilbert space framework. In Section \ref{sec:classical} we recall the classical AS one level method. 
In Section \ref{sec:nonsymmetric} we present the abstract analysis of symmetric methods. We first 
revisit the analysis for symmetric methods using 
projections and angles between sub-spaces. We generalize this analysis to nonsymmetric methods. In particular we apply this analysis to a special family of nonsymmetric methods. In Section \ref{sec:restricted} we define the restricted method that we analyze.  In Section 
\ref{sec:condfor-ras} we  write a condition number estimate of the restricted method of Section \ref{sec:restricted}. 

\section{A Hilbert space framework}\label{hilbertspace}

Let $H$ and $\widehat{G}$ be real Hilbert spaces   with inner products $a(\cdot, \cdot)$ and 
$b(\cdot,\cdot)$, respectively. The case of complex Hilbert spaces is similar. 
Consider  $R:H\to \widehat{G}$ to be a  bounded operator and {\color{black} denote 
by $\| R\|_{a,b}$ its operator norm}.
In domain decomposition methods literature $R$ is referred to as a \emph{restriction operator}. Introduce the transpose operator
 $R^{T,b}: {\color{black}\widehat{G}}\to H$ defined by 
\begin{equation}\label{eq:def:RT}
a(R^{T,b}x,v)= b(x, Rv) \quad  \mbox{ for all } v\in H, x\in \widehat{G}.
\end{equation}
{\color{black} Despite of the fact that $R^{T,b}$ depends on inner products $a$ and 
$b$, our notation makes  explicit only the dependence on $b$.}

Assume there is a (closed) subspace $G\subset {\color{black}\widehat{G}}$ such that 
\begin{equation}\label{ExtensionfromRestriction}
E= R^{T,b}|_G
\end{equation}
is easy to compute. The operator $E$ is known as an \emph{extension operator}. Note also that 
$E:G\to H$ and that we have $E^{T,b}:H\to G\subset{\color{black}\widehat{G}}$ with 
\begin{equation}
    E^{T,b}= \Pi_{G,b} R
\end{equation}
where $\Pi_{G,b}:{\color{black}\widehat{G}}\to G$ is the orthogonal projection on $G$ using the inner product $b$.
We want to study the operator 
\begin{equation}\label{symoperatorwithprojection}
EE^{T,b}=R^{T,b}\Pi_{G,b} R.
\end{equation} 
{\color{black} See Section \ref{sec:classical} for a particular example in the case of a one-level domain decomposition method.} 
This operator is clearly symmetric and non-negative {\color{black} definite} in the $b$ inner product. If we want $EE^{T,b}$ to be non-singular and since 
 $\mathcal{N}(E)=\mathcal{N}(R^{T,b})\cap G$  and $\mathcal{R}(E^{T,b})=\Pi_{G,b} \mathcal{R}(R)$ are $b-$orthogonal in $G$, we need  to be sure $ E^{T,b}= \Pi_{G,b} R$ is 1-1 or, equivalently, $E$ is onto.  A sufficient condition for the symmetric operator $EE^{T,b}=R^{T,b}\Pi_{G,b} R$ to be  invertible is given by the following lemma known as \emph{stable decomposition lemma} or \emph{Lion's lemma} in {\color{black}the} domain decomposition community. For the sake of completeness we show a detailed proof as it is usually presented in {\color{black}the} domain decomposition literature; see for instance  \cite[Chapter 2]{MR2104179} or \cite{MR2445659} and references therein. We note that we do not need to refer to the space ${\color{black}\widehat{G}}$ at this moment. Later we revise some of these inequalities in a more natural way to obtain a sharper estimate.
\begin{lemma}[Lions Lemma]\label{lions} Assume that there exists a bounded right inverse of $E$. That is, there exists a bounded operator $\widehat{E}:H \to G$ such that $E\widehat{E}v=v$ for all $v\in H$. Then, the mapping $EE^{T,b}:H\to H$ is non-singular.  Moreover, we have  
\[
 \| \widehat{E}\|^{-2}_{\color{black}a,b}\|v\|_{a}^2 \leq  \,a(v,EE^{T,b}v) =\| E^{T,b}v\|_{b}^2 \leq  \|E\|^2_{\color{black}b,a}\| v\|_a^2 
\]
for all $v\in H$.

\end{lemma}
\begin{proof}
Note that for $v\in H$ we have, 
\begin{eqnarray*}
\|v\|_a^2&=&a(v,v)\\&=& a(E\widehat{E}v,v)\\&=&b(\widehat{E}v,E^{T,b}v)\\&\leq&
b(\widehat{E}v, \widehat{E}v)^{1/2}
b(E^{T,b}v,E^{T,b})^{1/2}\\
&{\color{black}=}& \|\widehat{E}v\|_{b} \,a(v,EE^{T,b}v)^{1/2}\\
&\leq& \|\widehat{E}\|_{\color{black}a,b} \|v\|_a \,a(v,EE^{T,b}v)^{1/2}.
\end{eqnarray*}
Using this last inequality we obtain
\[
 \| \widehat{E}\|^{-2}_{\color{black}a,b}\|v\|_a^2 \leq  \,a(v,EE^{T,b}v).
\]
To obtain the upper bound we proceed as follows using properties of 
subordinated norm of operators,  
\begin{eqnarray*}
\|EE^{T,b}v\|^2_a&\leq &\| E\|^2_{\color{black}b,a}  \|E^{T,b}v\|_{b}^2\\
&=& \|E\|^2_{\color{black}b,a}   b(E^{T,b}v,E^{T,b}v) \\
&\leq& \|E\|^2_{\color{black}b,a}   a(v,EE^{T,b}v) \\
&\leq & \|E\|^2_{\color{black}b,a}\| v\|_a  \|EE^{T,b}v\|_a
\end{eqnarray*}
and therefore $\|EE^{T,b}v\|_a\leq \| E\|^2_{\color{black}b,a}\| v\|_a$. We also have, 
\begin{eqnarray*}
\,a(v,EE^{T,b}v)&\leq &\| v\|_a  \|EE^{T,b}v\|_a \leq \|E\|^2_{\color{black} b,a}\| v\|_a^2.
\end{eqnarray*}
This finishes the proof.
\end{proof}

\begin{remark}\label{remarkT}
Note that what it is needed is the existence of operator $\widetilde{E}:H\to G\subset{\color{black}\widehat{G}}$
such that $T=E\widetilde{E}:H\to H$ is invertible.  In this case we  have that
$\widehat{E}=\widetilde{E}T^{-1}$ is an stable right inverse of $E$.
\end{remark}
If, in addition, the extension operator $E$ comes from a 
restriction operator $R$, as in \eqref{ExtensionfromRestriction}, we can state the following corollaries. 
\begin{corollary}\label{corolaryR}
Let $R$ be a restriction operator such that 
$E=R^{T,b}|_G$. Assume that there exits a bounded operator $\mathcal{J}_{\color{black}b}:{\color{black}\widehat{G}}\to G\subset{\color{black}\widehat{G}}$ such 
that $R^{T,b} \mathcal{J}_bRv=v$ for all $v\in H$.  Then, the mapping $EE^{T,b}
=R^{T,b}\Pi_{G,b} R:H\to H$ is non-singular with
\[
 \| \mathcal{J}_{\color{black}b}R\|^{-2}_{\color{black}a,b}\|v\|_a^2 \leq  \,a(v,EE^{T,b}v) =\| E^{T,b}v\|_{b}^2 \leq  \|R\|^2_{\color{black}a,b}\| v\|_a^2 
\]
for all $v\in H$.

\end{corollary}

\begin{corollary}\label{corolaryR2}
Let $R$ be a restriction operator such that 
$E=R^{T,b}|_G$. Assume that there exits a bounded operator $\widehat{E}:H\to G\subset 
{\color{black}\widehat{G}}$ such that $R^{T,b} \widehat{E}v=v$ for all $v\in H$.  Then, the mapping $R^{T,b} R:H\to H$ is non-singular with
\[
 \| \widehat{E}\|^{-2}_{\color{black}a,b}\|v\|_a^2 \leq  \,a(v,R^{T,b}Rv) =\| Rv\|_{b}^2 \leq  \|R\|^2_{\color{black}a,b}\| v\|_a^2 
\]
for all $v\in H$.

\end{corollary}

{\color{black} Let $L:H\mapsto \mathbb{R}$ be a bounded linear functional on $H$}. Denote by $u\in H$  the solution of the following variational equation, 
 \begin{equation}\label{problem}
 a(u,v)=L(v) \quad \mbox{ for all } v\in H.
 \end{equation}
Assuming that $E$ is easy to compute. We see that, 
 for the solution $u$, 
$E^{T,b}u$ is possible to compute using this variational equation (without explicitly knowing or computing the function $u$). In fact, we have  
 \begin{equation}\label{computationofET}
 b(E^{T,b}u, \phi)=a(u, E\phi)=L( E\phi) =LE\phi \quad \mbox{ for all } \phi\in G.
 \end{equation}
This equation might be easier to solve numerically than the original problem. 
Therefore, we can alternatively compute the solution of (\ref{problem}) by iteratively 
solving the equation, 
\begin{equation}\label{preconditioned}
EE^{T,b} u = \widetilde{L} 
\end{equation}
{\color{black} where the right hand side $\widetilde{L}$ can be computed by solving \eqref{computationofET} 
and applying the extension operator $E$}.  When implementing an iterative method to solve \eqref{preconditioned}, in 
each iteration we have to apply the operator $EE^{T,b}$ to a residual vector, say 
$r$. More precisely, we have to
\begin{enumerate}
\item Compute $x=E^{T,b}r$, this can be done by solving the equation
 \begin{equation}\label{computationofET2}
 b(x, \phi)=a(r, E\phi) \quad \mbox{ for all } \phi\in H.
 \end{equation}
  In terms of the restriction operator $R$ we have $x=\Pi_{G,b} Rr$.
 \item Compute $s=Ex=EE^{T,b}r$ by applying the extension operator $E$, that is 
  $s=R^{T,b}\Pi_{G,b} R r=E\Pi_{G,b} R r$ {\color{black}which} is assumed possible and numerically efficient to compute.

\end{enumerate} 

{\color{black}
The practicality of using the iteration
depends on the possibility to inexpensively compute the right hand side $\widetilde{L}$ and the 
number of iterations needed until convergence.
The condition of $EE^{T,b}$ give us some information about the difficulty in solving the corresponding equation. 
In particular, since  $EE^{T,b}$ is symmetric (and positive-definite as we will see later),   PCG could be applied. In this case  the performance of the iterative procedure depends 
on the condition number of the associated operator equation. If we use the spectral 
condition number of the operator $EE^{T,b}$, we see from Lemma \ref{lions} that 
\[
\kappa_{spectral}(EE^{T,b})\leq \|\widehat{E}\|^2_{a,b} \|E\|^2_{b,a}.
\]
Then, for some iterative methods such as $CG$, the number of iterations for solving the equation (\ref{preconditioned}) (up to a desired tolerance) will depend on 
$\|\widehat{E}\|^2_b \|E\|^2_b$. }

In general, the condition number of an operator $T:H\to H $ is defined by  
\[
\kappa(T)= \|T^{-1}\|_{\color{black}a,a} \|T\|_{\color{black}a,a}.
\]
{\color{black} For general iterative methods such as GMRES that could be applied to non-symmetric problems, bounds for the condition number alone are not enough for the convergence of the method. In this case we need information about the distribution of the eigenvalues. Nevertheless, condition number bounds may lead to further understanding of iterative methods applied to non-symmetric methods.}

\section{Classical additive method for Laplace equation}
\label{sec:classical}
In this section we use the Hilbert space framework above to review the 
analysis of the classical additive method.
As usual, we consider a subdomain $D$ with a non-overlapping partition of the domain into subdomains $\{ D_\ell \}_{\ell=1}^{N_S}$. By enlarging these subdomains an specific width $\delta$ we obtain and overlapping decomposition 
$\{ \mathcal{O}_\ell \}_{\ell=1}^{N_S}$. For more details 
see \cite{MR2104179}.

Let $H=H^1_0(D)$,  
${\color{black}\widehat{G}}=\times_{\ell=1}^{N_S} 
H^1(\mathcal{O}_\ell)$
and $G=\times_{\ell=1}^{N_S} 
H^1_0(\mathcal{O}_\ell)\subset {\color{black}\widehat{G}}$. 
In this case consider 
\[
a(u,v)=\int_D \nabla u\nabla v \quad \mbox{ for all } 
u,v\in H. 
\]
Denoting by $\{v_\ell\}$ the elements of $
{\color{black}\widehat{G}}$, we 
define for {\color{black}$\{v_\ell\}, \{w_\ell\}\in 
\widehat{G}$}
\[
b(\{v_\ell\},{\color{black}\{w_\ell\}})=
\sum_{\ell=1}^{N_S}\int_{\mathcal{O}_\ell} \nabla 
v_\ell\nabla {\color{black}w_\ell}+\int_{\partial\mathcal{O}_\ell} v_\ell{\color{black} w_\ell}
=\sum_{\ell=1}^{N_S} b_\ell(v_\ell,w_\ell)+b_\ell^{\partial}(v_\ell,{\color{black}w_\ell})
\]
where we have put  
$b_\ell(v_\ell,w_\ell)=\int_{\mathcal{O}_\ell} 
\nabla v_\ell\nabla w_\ell $ and $b_\ell^{\partial}(v_\ell,w_\ell)=\int_{\partial\mathcal{O}_\ell} v_\ell w_\ell$.

\begin{remark}[Norm boundary term]
The {\color{black}role} of $b_\ell^{\partial}$ is not essential and can be replaced by any other bilinear form that 
vanish for functions on $G$ which makes $b$ {\color{black} a positive definite bilinear form on ${G}$}.
\end{remark}

Introduce also $R:H\to {\color{black}\widehat{G}}$ defined by 
$R u = \{ u|_{\mathcal{O}_\ell}\}$. Equation  \eqref{eq:def:RT} defining $R^{T,b}$ corresponds to 
\[
\int_D \nabla R^{T,b}{\color{black}\{v_\ell\}}\nabla z=
b(\{v_\ell\},\{z_\ell|_{\mathcal{O}\ell}\})\quad \mbox{ for all }
z\in H.
\]
{\color{black}This definition} implies that $E: G\to H$ {\color{black}
where} $E=R^{T,b}|_G$, is given by
\[
E \{v_\ell\} = \sum_{\ell=1}^{N_S} E_\ell v_\ell
\]
where $E_\ell$ is the extension by zero outside $\mathcal{O}_\ell$  operator. To see this note that 
\[
\int_D \nabla E\{v_\ell\}\nabla z=
\sum_{\ell=1}^{N_S}\int_{\mathcal{O}_\ell} \nabla 
v_\ell\nabla z
\quad \mbox{ for all }
z\in H.
\]

We have that $E^{T,b} : H\to G$ is given by 
\[
b(E^{T,b}u,\{v_\ell\})= a(u, E\{v_\ell\}).
\]
Note that $E^{T,b}u=\{E^{T,b}_\ell u\}$ where $E^{T,b}_\ell u$ solves the local equation
\[
\int_{\mathcal{O}_\ell} \nabla E^{T,b}_\ell u\nabla v_\ell = 
\int_{D} \nabla u \nabla E_\ell v_\ell \mbox{ for all } v_\ell \in H^1_0(\mathcal{O}_\ell).
\]
Observe that $E_\ell^{T,b}$ can be obtained by solving a local problem. {\color{black} We define the additive method  by $P_{add}EE^{T,b}$}. Therefore,
\[ 
P_{add}u=EE^{T,b} u = \sum_{\ell=1}^{N_S} E_\ell E^{T,b}_\ell u =
\left(\sum_{\ell=1}^{N_S} P_\ell \right) u.
\]
{\color{black}Here we  denote $P_{\ell}=E_\ell E^{T,b}_\ell $.}
{\color{black}The existence of a right inverse can be stated as follows as it is common in domain 
decomposition literature. In fact, to obtain the stable inverse assume:\\
\begin{itemize}
\item {\bf Stable decomposition}: There exists a constant $C_E$ such that  for all $v\in V$ there exist $v_\ell\in H^1_0(
\mathcal{O}_\ell)$, $\ell=1,\dots, N_S$ such 
that $v=\sum_{\ell=1}^{N_S} E_\ell v_\ell$ and 
\begin{equation}\label{eq:establedecomposition}
\sum_{\ell=1}^{N_S} b_\ell(v_\ell,v_\ell)\leq C_E^2 a(v,v).
\end{equation}

\item {\bf Strengthened  Cauchy inequalities}. There exits a matrix ${\mu}=(\mu_{\ell k})_
{\ell,k}$ with $\mu_{\ell k}\leq 1$ and such that 
\[
a(E_\ell v_\ell, E_k u_k ) \leq \mu_{\ell k} b_\ell(v_\ell,v_\ell)^{1/2}b_k(v_k,v_k)^{1/2}.
\]
\end{itemize}
The stable decomposition assumption clearly implies the existence of $\widehat{E}$ and $\|\widehat{E}\|_{b,a}\leq C_E$. In fact, 
$\widehat{E}v=\{ v_\ell\}$ where the functions $v_\ell$ are the ones given by the 
stable decomposition assumption.}\\
By using bilinearity and vector Chauchy inequalities, this clearly implies that 
\begin{eqnarray*}
a(E\{v_\ell\},E\{v_k\})&=&\sum_{\ell,k}^{N_S} a(E_\ell  u, E_{k} u)\\
&\leq &\sum_{\ell,k} \mu_{\ell k} b_\ell(v_\ell, v_\ell)^{1/2} b_k(v_k, v_k)^{1/2}\\
&\leq& \rho(\mu) \left(\sum_{\ell=1}^{N_S} b_\ell(v_\ell,v_\ell) \right)
\leq \rho(\mu)  b\left(\{v_\ell\},\{v_\ell\}\right),
\end{eqnarray*}
where $\rho({\color{black}\mu})$ is the spectral radius of the matrix $\mu$ above. Then $\|E\|_{\color{black}b,a}\leq \sqrt{ \rho(\mu)}$.  { \color{black}Using this and the fact that }$\|\widehat{E}\|_{\color{black}b,a}\leq C_E$  in 
Lemma \ref{lions} we have the following result.
\begin{corollary}
For all $u\in V$ we have 
\begin{equation}
C_E^{-2}a(u,u)\leq a(P_{add}u,u)\leq \rho(\mu) a(u,u)
\end{equation}
\end{corollary}
where $C_E$ is the stable decomposition constant and $\rho({\color{black}\mu})$ is the spectral radius of the matrix {\color{black}$\mu$} above.

\begin{remark} \label{remarkclassicalSD}
Let us consider the case of the one level additive method setting. More levels can be analyzed in a similar way.  In the one level setting, with original domains of diameter 
$\tau$ and overlap of size $\delta$ a usual bound for $C_E$ is given as follows by constructing a stable decomposition as follows; see 
\cite{MR2104179}. Start by constructing cut functions $\eta_\ell$ such that
\begin{equation}\label{eq:def:cutoff}
\eta_\ell(x)=1 \mbox{ for all } x\in D_\ell, \mbox{ } \eta_\ell(x)=0 \mbox{ for all } x\not\in \mathcal{O}_\ell, \mbox{ }  
|\nabla \eta_\ell(x)|\leq C_{cut} \frac{1}{\delta}.
\end{equation}
Define the partition of unity function 
\[
\chi_\ell= \frac{\eta_\ell}{\eta} \quad \mbox{ where }
\eta= \sum_{\ell=1}^{N_S}\eta_\ell .
\]
We see that 
$$
\chi_\ell(x)=1 \mbox{ for all } x\in D_\ell\setminus 
\cup_{\ell'\not=\ell}\mathcal{O}_{\ell'}, \mbox{ } \chi_\ell(x)=0 \mbox{ for all } x\not\in \mathcal{O}_\ell,$$
and therefore
\begin{equation}\label{eq:def:pu}
|\nabla\chi_\ell(x)|\leq C_{pu} \frac{1}{\delta}.
\end{equation}
We should have $C_{cut}\leq C_{pu}$. Then define {\color{black}the stable right inverse by}
$\widehat{E}u=\{ \chi_\ell u\}$. Denoting 
\begin{equation}\label{eq:def:nu} 
Neigh(j)=\left\{ j : \mathcal{O}_j\cap \mathcal{O}_i \not = \emptyset\right\}
\mbox{ and } 
\nu =\max_{j} \# Neigh(j)
\end{equation}
we have (see \cite{MR2104179})
\[
C_E^2\leq  \nu C_{pu}\left(1+\frac{1}{\tau\delta}\right).
\]
The {\color{black}norm of $E$} and $\rho(\mu)$ are bounded by
\[
\|E\|^2_{\color{black}b,a}\leq {\color{black}  \rho(\mu)} \leq  \nu.
\]

\end{remark}

In case we want to approximate the solution of  problem (\ref{problem}), 
we see that $u$ also solves the operator equation, 
\begin{equation}\label{preconditioned1Ladd}
EE^{T,b} u = \left(\sum_{\ell=1}^{N_S}  P_\ell\right) u=\widetilde{L}.
\end{equation}
{\color{black}Here $\widetilde{L}$ is obtaining by 
$E$-assembling  the solutions of the local problems, 
\[
b_\ell(E_\ell^{T,b} u,v_\ell)=L(E_\ell v_\ell) \mbox{ for all } v_\ell\in H^1_0(\mathcal{O}_\ell).
\]
The linear system 
(\ref{preconditioned1Ladd}) is usually better conditioned that the original linear system 
$a(u,v)=L(v)$.}

\section{Nonsymmetric methods obtained 
by changing restrictions and the inner product}\label{sec:nonsymmetric}
We use the Hilbert space framework introduced in 
Section \ref{hilbertspace}. 
Recall that we have 
$H$ and $G\subset {\color{black}\widehat{G}}$ Hilbert spaces with inner products $a(\cdot, \cdot)$ and 
$b(\cdot,\cdot)$, respectively. We also used  the bounded restriction operator
 $R:H\to {\color{black}\widehat{G}}$ {\color{black}for the definition of the extension operator $E$}.

{\color{black} To develop a framework for non-symmetric methods} we additionally introduce a second bi-linear form 
 $c(\cdot,\cdot)$ defined on ${\color{black}\widehat{G}}$.
 Let  us introduce a possibly different and bounded restriction 
 operator  $S: H\to {\color{black}\widehat{G}}$
 and the transpose $S^{T,c}$ defined analogously to \eqref{eq:def:RT} by 
 \begin{equation}\label{eq:def:RTc}
a(S^{T,c}\phi,v)= c(\phi, Sv) \quad  \mbox{ for all } v\in H.
\end{equation}

Define $F=S^{T,c}|_G$ as a second extension operator. 
As before, assume that there is an stable left-inverse for $F$, say $\widehat{F}$ such that $F\widehat{F}v=v$
for all $v\in H$ (that is, $\widehat{F}$ is bounded in the 
$c$ and $b$ inner product norms). We can then conclude about $F$ and $F^{T,c}$ similar inequalities than the given before in the case $c$ is symmetric and positive definite.  In particular $\color{black}F^{T,c}$ is a bijective application from $H$ onto 
$\mathcal{R}(F^{T,c})$.

\begin{corollary}\label{corolaryRc}
Let $S$ be a restriction operator {\color{black}and}  
$F=S^{T,c}|_G$. Assume that there exits a bounded right inverse of  $F=S^{T,c}|_G$, 
say $\widehat{F}$.  Then, the mapping
$FF^{T,c}=R^{T,c} \Pi_{G,c} R$ is non-singular with
\[
 \| \widehat{F}\|^{-2}_{\color{black}a,c}\|v\|_a^2 \leq  \,a(v,FF^{T,c}v) =\| F^{T,c}v\|_c^2 \leq  \|F\|^2_{\color{black}c,a}\| v\|_a^2 
\]
and 
\[
 \| \widehat{F}\|^{-2}_{a,c}\|v\|_a^2 \leq  \,a(v,S^{T,c}Sv) =\| Sv\|_c^2 \leq  \|S\|^2_{\color{black}c,a}\| v\|_a^2 
\]
for all $v\in H$.

\end{corollary}

We want to study the nonsingularity of the operator $FE^{T,b}:H\to H$. Note that 
\begin{equation}\label{eq:FETbSTcPiGbR}
    FE^{T,b}= S^{T,c} \Pi_{G,b}R.
\end{equation}
See \eqref{symoperatorwithprojection}. This operator is  nonsymmetric for general bi-linear forms $b$ and $c$. This is due to the fact that $\Pi_{G,b}$ might not be symmetric in the $c$ bilinear form.  

{\color{black}
\begin{example}\label{especial}
 As a particular case of our general construction we can put $S=R$. In this case, 
$F=R^{T,c}|_G$ and  $F^{T,c}=\Pi_{G,c} R$. We can then obtain the operators 
 \begin{equation}\label{eq:FETbRTcPiGbR}
FE^{T,b}= R^{T,c} \Pi_{G,b} R  \quad\mbox{ and }\quad 
EF^{T,c}= R^{T,b} \Pi_{G,c} R.
 \end{equation}
 In Section \ref{specialanalysis} we obtain condition number bounds for operator $FE^{T,b}$ in this example. This will allows us to write condition  number estimates for a non-symmetric method that uses local problems similar to the local problems of the OBDD in 
 \cite{kimn2007obdd}. See Section \ref{sec:restricted}.
\end{example}

\begin{example}\label{original}
Another particular case is when $c=b$ and $S=M^{T,b}R$ where $M:\widehat{G}\to \widehat{G}$ is a bounded operator. Then  
$F=S^{T,b}|_G = (R^{T,b}M)|_G$ and  $F^{T,b}=\Pi_{G,b} S=\Pi_{G,b} M^{T,b}R$. We can write 
 \begin{equation}\label{eq:originalras1}
FE^{T,b}=  R^{T,b}M \Pi_{G,b} R
 \end{equation}
and also 
 \begin{equation}\label{eq:originalras2}
EF^{T,b}=  R^{T,b}\Pi_{G,b} M^{T,b} R.
 \end{equation}
If $M(G)\subset G$ we are left with $F=EM$ and $FE^{T,b}=  E M E^{T,b}$. In addition if $M|_G:G\to G$ has a bounded inverse, 
$EF^{T,b}= F M^{-1} F^{T,b}.$\\
Therefore, the spectral properties of the resulting operators 
\eqref{eq:originalras1} and \eqref{eq:originalras2} will depend on the spectral properties of $M$. A case we explicitly mention is the case where $M$ is defined as point-wise multiplication operator in the context of Section \ref{sec:classical}.

Let us consider the example of Section \ref{sec:classical} and 
select a cut-off functions $\eta^{cut}=\{\eta^{cut}_\ell\}\in {\color{black}\widehat{G}}$. Define the operator 
$M: {\color{black}\widehat{G}}\to {\color{black}\widehat{G}}$ by 
\[
M\{v_\ell\} = \{ \eta_\ell^{cut} v_\ell\}.
\]
This operator is not symmetric with respect to the $b$ blinear form. In Section \ref{sec:classical}, $E$ is the extension by zero operator and then the extension operator $F=EM$ corresponds to extending by zero after a pointwise multiplication by $\eta_\ell^{cut}$ in each overlapping subdomain $\mathcal{O}_\ell$. According to \eqref{eq:originalras1} and
recalling the definition of 
$E^{T,b}$ in Section \ref{sec:classical}  we see that  $FE^{T,b}=EME^{T,b} =\sum_{\ell=1}^{N_S} E_\ell \eta_\ell^{cut} E^{T,b}_\ell$. It needs the solution of local problems, these local solution are then point-wise multiplied by the cut functions and after that an extension by zero and addition follows. This is exactly the RAS preconditioned operator as introduced by Cai and Sarkis in \cite{CaiSarkis}. On the other hand, according to \eqref{eq:originalras2}, 
$EF^{T,b}u=  R^{T,b}\Pi_{G,b} M^{T,b} Ru=  R^{T,b}\Pi_{G,b} M^{T,b} \{(u|_{\mathcal{O}_\ell})\}$. To compute $M^{T,b} \{(u|_{\mathcal{O}_\ell})\}$ we need to solve problem \eqref{reveq:strong:cutressolve} and then extend the local solutions by zero. See 
\eqref{reveq:strong:cutressolve}  and Remark \ref{rev:remarkclarify} later in Section \ref{sec:restricted}. The local problem \eqref{reveq:strong:cutressolve} is  the local 
problem used in the OBDD method in \cite{kimn2007obdd}.\\

A successful application of the  abstract analysis developed in this paper (see Section \ref{sec:abstract})  to the analysis of the operators in this example 
was not obtained here. This is a topic of ongoing research. The main issue is that we need results on the angle between subspaces 
$\mathcal{R}(F^{T,b})=M\mathcal{R}(E^{T,b})$ and 
$\mathcal{R}(E^{T,b})$ that  are not available (see Section \ref{sec:normproj}). In Section   \ref{sec:restricted} we are able to bound the condition number of a one level method that uses local problems similar to the local
problems of the OBDD. See Remark \ref{rev:remarkclarify}.
\end{example}
}

\begin{remark}[Perturbation theory]
Note {\color{black}that} we can write
$$
EF^{T,b} = EE^{T,b} + E(F-E)^{T,b}= EE^{T,b} +J
$$
where $J= E(F-E)^{T,b}$ is a perturbation of $EE^{T,b}$ of size 
$\|J\|_{\color{black}a,a}=\| E(F-E)^{T,b}\|_{\color{black}a,a}$. Several results can be {\color{black}pursued} of the type: If 
$\| J\|$ is {\color{black}small enough}, then the operator $EF^{T,b}$ will be invertible and it is possible to estimate its condition number. {\color{black}We think that this approach is not} practical for analyzing domain decomposition  methods. 
\end{remark}

\section{Condition number estimates using norms of projections}\label{sec:abstract}
In this section we present a different analysis that may turn useful when estimating condition number of preconditioned operators (not-necessarily constructed by a domain decomposition design).
We present a series of projection arguments to study nonsymmetric methods. 
As presented earlier, the idea is to estimate the condition number of an operator of the form
$EF^{T,b}$ where $E,F:G\to H$ are different extension operators. In particular we are able to bound the condition numbers for the family of 
nonsymmetric methods presented in Section \ref{sec:nonsymmetric} where the extension operators are defined from restriction operator from $H$ to a bigger space ${\color{black}\widehat{G}}\supset G$. Before going  
to nonsymetric methods we review norms of projections.

 \subsection{\color{black}Norms of projections}\label{sec:normproj}
We need the following definitions
and results; 
see \cite{Daniel, 2p, galantai2013projectors}. Let $X$ and $Y$ be subspaces of 
$G$ (or ${\color{black}\widehat{G}}$). Introduce the minimal angle between subspaces $X$ and $Y$ with respect to the inner product $b$, 
$\theta_b(X,Y)$ as 
\begin{equation}
\sin(\theta_b(X,Y)) = \inf_{\|x\|_{\color{black}b}=1, x\in X} \mbox{dist}_b(x,Y)=
\inf_{\|x\|_{b}=1, x\in X} \sqrt{1-\|\Pi_Yx\|^2_b}.
\end{equation}
Equivalently, we have $\sin(\theta_b(X,Y))^2= 1-\cos^2( \theta_b(X,Y))$ where
\begin{equation}
\cos( \theta_b(X,Y)) = \sup_{x\in X, y\in Y} \frac{b(x,y)}{\|x\|_{b}\|y\|_{b}}=
\|\Pi_X\Pi_Y\|_{\color{black}b,b}=\|
{\color{black}\Pi_Y\Pi_X}\|_{\color{black}b,b}.
\end{equation} 
Still equivalent, we have, 
\begin{equation} \label{angle}
\sin(\theta_b(X,Y)) = \frac{1}{\| Q(X,Y) \|_{\color{black}b,b}}
\end{equation}
where $Q(X,Y)$ is the (oblique) projection on $X$ along $Y$. {\color{black} Note that if $Y=X^{\perp,b}$ then 
$Q(X,Y)=\Pi_{X,b}$}. Introduce the maximal angle between subspaces $X$ and $Y$, 
$\Theta_b(X,Y)$ as 
\begin{eqnarray*}
\sin(\Theta_b(X,Y)) &=& \sup_{\|x\|_{a,b}=1, x\in X} \mbox{dist}_b(x,Y) \\&=&
\sup_{\|x\|_{b}=1, x\in X} \sqrt{1-\|\Pi_{Y,\color{black} b}x\|^2_b}=
\|\Pi_{X,\color{black} b}-\Pi_{Y,\color{black} b}\|_{\color{black}b,b}.
\end{eqnarray*}
Equivalently we have $\sin(\Theta_b(X,Y))^2= 1-\cos^2( \Theta_b(X,Y))$ where
\begin{equation}
\cos( \Theta_b(X,Y)) = \inf_{x\in X, y\in Y} \frac{b(x,y)}{\|x\|_{b}\|y\|_{b}}.
\end{equation} 
We also have, 
\[
 \theta_b(X,Y) +\Theta_b(X,Y^{\perp,b}) = \frac{\pi}{2},
\]
and 
\begin{equation}\label{Thetatheta}
\sin(\theta_b(X,Y)) = \cos(\Theta_b(X,Y^{\perp,b})).
\end{equation}

\subsection{General nonsymmetric method analysis using projections}

Let $F:G\to H$ be a second extension operator. We want to study the operator $FE^{T,b}$. See Figure \ref{fig2} for an illustration.

\begin{figure}
    \centering
    \includegraphics[width=0.85\textwidth]{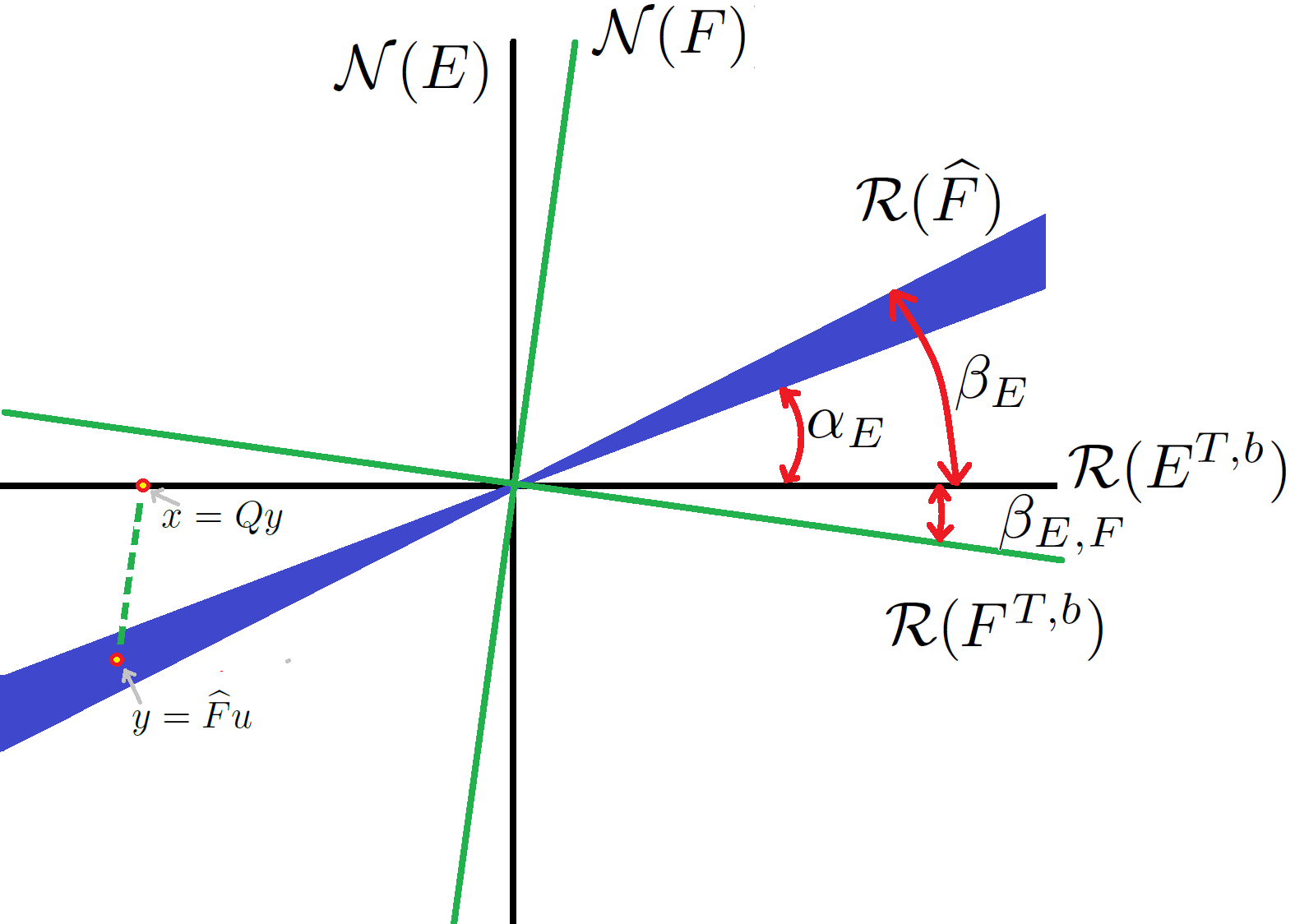}
    \caption{Illustration of subspaces of $G$. In order to illustrate angles we 
    picture $\color{black}\mathcal{R}(\widehat{F})$ as a cone. We also illustrate the procedure presented in the 
    proof of Theorem \ref{theorem:nonsymetricgeneral} 
    and the oblique projection $Q=Q(\mathcal{R}(E^{T,b}),\mathcal{N}(F))$.}\label{fig2}
\end{figure}
 
\begin{theorem}\label{theorem:nonsymetricgeneral} Consider extensions operators $E$ and $F$ with stable right 
inverse $\widehat{E}$ and $\widehat{F}$, respectively. Assume the boundedness of  $Q=Q(\mathcal{R}(E^{T,b}),\mathcal{N}(F))$,   the oblique projection onto  $\mathcal{R}(E^{T,b})$ {\color{black}along}
$\mathcal{N}(F)$. 
 Then, the operator $FE^{T,b}:H\to H$ is invertible. Moreover, 
$$\|(FE^{T,b})^{-1}\|_{\color{black}a,a}\leq  \|Q|_{\mathcal{R}(\widehat{E})}\|_{\color{black}b,b}
\|\widehat{E}\|_{\color{black}a,b}\|\widehat{F}\|_{\color{black}b,a}.$$

\end{theorem}
\begin{proof}
We solve the equation
$$
FE^{T,b} w=u.
$$

Let $u\in H$ be given.  
\begin{enumerate}
\item Define $y=\widehat{F}u\in G$. Then we readily see that $Fy=u$.  By assumption 
we then have 
\begin{equation}\label{normofy2}
\|y\|_{b}\leq \|\widehat{F}\|_{\color{black}a,b} \|u\|_a.
\end{equation}
\item Construct 
$x\in \mathcal{R}(E^{T,b})$ such that $Fx=Fy=u$.   Here we use the oblique projection 
$Q=Q(\mathcal{R}(E^{T,b}),\mathcal{N}(F))$. 
See Figure \ref{fig2}. In fact, $x=Qy$. By definition of the projection $Q$ we have 
$F(y-x)=F(y-Qy)=0$ so that $Fx=Fy$. We have, 
\begin{equation}\label{op}
\|x\|_{b}\leq \|Q|_{\mathcal{R}(\widehat{E})}\|_{\color{black}b,b} \|y\|_{b}.
\end{equation}

\item Take $w\in H$ such that 
$E^{T,b}w=x$. In fact, $w=\widehat{E}^{T,b} x$. This $w$ is the solution of the equation above since we have
$EE^{T,b}w=Fx=Fy=u$. We  can bound 
\begin{equation}\label{normofw2}
\|w\|_{a}\leq \|\widehat{E}\|_{a,b}\| x\|_{b}.
\end{equation}
\end{enumerate}
By combining the estimates in (\ref{normofy2}), (\ref{op}) and (\ref{normofw2}) above we finish the proof.
\end{proof}

We can now give a bound for the condition number of the operator $\color{black}
FE^{T,b}$.
\begin{corollary} We have
$$\kappa(FE^{T,b}) =\|FE^{T,b}\|_{\color{black}a,a}\|(FE^{T,b})^{-1}\|_{\color{black}a,a}
\leq   \|\widehat{F}\|_{\color{black}a,b}\|F\|_{b,a}\|Q|_{\mathcal{R}(\widehat{E})}\|_{\color{black}b,b} \|E\|_{b,a}\|\widehat{E}\|_{\color{black}a,b}.$$
\end{corollary}
We also have the following corollary.
\begin{corollary} If 
$\mathcal{N}(F)$ is orthogonal to $\mathcal{R}(E^{T,b})$ (or $\mathcal{N}(F)\subset
\mathcal{N}(E)$) then 
$$\kappa(FE^{T,b}) =
\|\Pi_{\mathcal{R}(E^{T,b}), \color{black}b}
|_{\mathcal{R}(\widehat{E})}
\|_{\color{black}b,b}\|FE^{T,b}\|_{\color{black}a,a}\|(FE^{T,b})^{-1}\|_{\color{black}a,a}
\leq   \|F\|_{\color{black}b,a}\|E\|_{\color{black}b,a}\|\widehat{F}\|_{\color{black}a,b}\|\widehat{E}\|_{\color{black}a,b}.$$
\end{corollary}
Finally, our result generalizes the analysis of the symmetric method in the sense that we have the following corollary. 
\begin{corollary} If $F=E$ we have, 
$$\kappa(EE^{T,b}) =\|EF^{T,b}\|_{\color{black}a,a}\|(EF^{T,b})^{-1}\|_{\color{black}a,a}
\leq  \cos(\alpha_E) \|E\|^2_{\color{black}b,a}\|\widehat{E}\|^2_{ a,b} \leq \|E\|^2_
{\color{black}b,a}\|\widehat{E}\|^2_{\color{black}a,b}$$
where $ \alpha_E$ is the minimal angle between subspaces 
$\mathcal{R}(\widehat{E})$ and $\mathcal{R}(E^{T,b})$, that is, 
\begin{equation}\label{reveq:dev:alphaE}
\alpha_E=\theta_b\left(\mathcal{R}(\widehat{E}),\mathcal{R}(E^{T,b})\right).
\end{equation}
\end{corollary}

\begin{proof} Denote by $\Pi_{\mathcal{R}(E^T)}|_{\mathcal{R}(\widehat{E})}$ the restriction of $\Pi_{\mathcal{R}(E^{T,b})}$ to $\mathcal{R}(\widehat{E})$. 
Observe that (see \cite{Daniel, 2p, galantai2013projectors})
\begin{eqnarray}
\|\Pi_{\mathcal{R}(E^{T,b})}|_{\mathcal{R}(\widehat{E})}\|_{\color{black}b,b} &=&
\sup_{\|x\|_{b}=1, x\in \mathcal{R}(\widehat{E})} \|\Pi_{\mathcal{R}(E^{T,b}), \color{black}b}x\|_{b}\\
&=&
\sup_{\|x\|_{b}=1, x\in \mathcal{R}(\widehat{E})} \sqrt{ 1 -
\|\Pi_{\mathcal{N}(E),\color{black}b}x\|^2_b}\\
&=&\sin \left(\Theta_b\left(\mathcal{R}(\widehat{E}),\mathcal{N}(E)\right)\right)\\
&=&\cos \left(\theta_b\left(\mathcal{R}(\widehat{E}),\mathcal{R}(E^{T,b})\right)\right).
\end{eqnarray}
See Figure \ref{fig1} for an illustration of this case.
See \cite{Daniel, 2p, galantai2013projectors} for more details and related results on oblique projections.

\end{proof}

\begin{figure}    \centering
    \includegraphics[width=0.85\textwidth]{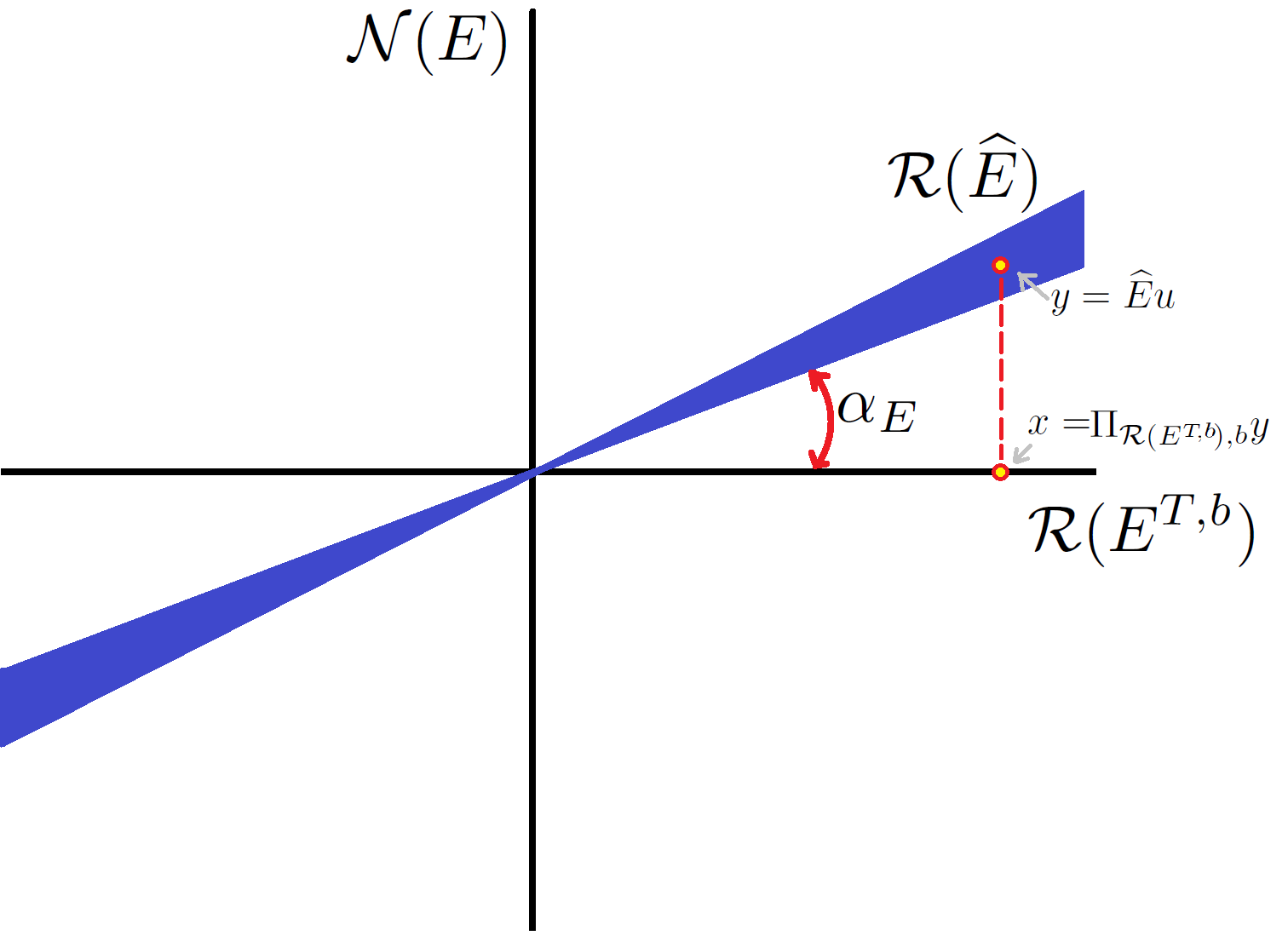}
    \caption{Illustration of subspaces of $G$. In order to illustrate angles we 
    picture $\mathcal{R}(\widehat{E})$ as a cone. We also illustrate the projection 
    $x=\Pi_{\mathcal{R}(E^{T,b}),b}y$.}\label{fig1}
\end{figure}

There is another interesting observation that is useful for the analysis and it is worth {\color{black}saying} as a result before we move on. 

\begin{lemma}
The operator $Q_E=\widehat{E} E$ is a projection on $\mathcal{R}(\widehat{E})$  along $\mathcal{N}(E)$.  Analogously, the operator 
$Q_E^{T,b}=E^{T,b}\widehat{E}^{T,b}$ is a projection on $\mathcal{R}(E^{T,b})$ and along
$\mathcal{N}(\widehat{E}^c)$. 
\end{lemma}
Using this lemma we can study the relative position of subspaces of interest. For instance we have,
\begin{equation}\label{angleEtilde}
\cos(\Theta_b(\mathcal{R}(\widehat{E}) , \mathcal{R}(E^{T,b}))=\sin(\theta_b(\mathcal{R}(\widehat{E}) , \mathcal{N}(E)) = \frac{1}{\| Q_E \|_{\color{black}a,a}} \geq 
\frac{1}{\|\widehat{E}\|_{\color{black}a,b} \|E\|_{\color{black}b,a}}.
\end{equation}
See Figure \ref{fig1}.
\begin{remark}\color{black}
Note that 
$\Pi_{\mathcal{R}(E^{T,b}),\color{black} b}|_{\mathcal{R}(\widehat{E})}$ and 
$Q_E|_{\mathcal{R}(E^{T,b})}$ are inverse to each other. 
That is 
$\Pi_{\mathcal{R}(E^{T,b}),\color{black}b}
Q_E x=x$ for all 
$x\in \mathcal{R}(E^{T,b})$ and 
$Q_E\Pi_{\mathcal{R}(E^{T,b}),\color{black}b}
 y=y$ for all 
$y\in \mathcal{R}(\widehat{E})$.
\end{remark}

{\color{black}We now come back to the general  case where $F\not=E$}. In practice we have to estimate the norms $ \|F \|_{\color{black}b,a}, \|E\|_{\color{black}b,a},\|\widehat{E}\|_{\color{black}a,b}$, 
$\|Q\|_{\color{black}b,b}$ and  $\|\widehat{F}\|_{\color{black}a,b}$. The norm  $\|\widehat{E}\|_{\color{black}b,a}$ it is usually required in 
symmetric methods. The norm $\|\widehat{F}\|_{\color{black}a,b}$ corresponds to the new extension operator
 used to obtain the nonsymmetric method.  
The norm    $\|Q|_{\mathcal{R}(\widehat{E})}\|_{\color{black}b,b}$ corresponds to a compatibility 
 of them both extension operators.
 
   {\color{black}There are  ways to try} to estimate $\|Q|_{\mathcal{R}(\widehat{E})}\|$ that may lead to 
 different analysis for nonsymmetric methods. See \cite{Daniel, 2p, galantai2013projectors}.
  In case it is technically difficult to get a bound for $\|Q|_{\mathcal{R}(\widehat{E})}\|$, we can use the fact that 
$$
Q y = Q \Pi_{\mathcal{R}(F^T),\color{black} b} y
$$
for all $y\in {\color{black}\widehat{G}}$ and therefore we can use the bound
 \begin{equation}
 \|Q|_{\mathcal{R}(\widehat{F})}\|_{\color{black}b,b}\leq 
 \|Q|_{\mathcal{R}(F^T)} \|_{\color{black}b,b} 
 \|\Pi_{\mathcal{R}(F^T),\color{black}b}|_{{\mathcal{R}(\widehat{F})}} \|_{\color{black}b,b}
\leq  \|Q\|_{\color{black}b,b} \cos(\alpha_F)=
 \frac{\cos(\alpha_F)}{\cos(\beta_{E,F})}
 \end{equation}
where $\beta_{E,F}$ is the maximal angle between subspaces 
$\mathcal{R}(F^{T,b})$ and $\mathcal{R}(E^{T,b})$, that is, 
\begin{equation}\label{eq:def:betaf}
\beta_{E,F}=\Theta_b\left( \mathcal{R}(F^{T,b}), \mathcal{R}(E^{T,b})\right)
=\Theta_b\left( \mathcal{N}(F), \mathcal{N}(E)\right).
\end{equation}
See Figure \ref{fig2}. Here we used  
 (\ref{angle}) to obtain,
\[
\|Q\|^{-1}_{\color{black}b,b} = 
\sin\left(
\theta_b\left(\mathcal{R}(F^{T,b}),\mathcal{N}(E)\right)
\right)=
  \cos\left(
\Theta_b\left(\mathcal{R}(F^{T,b}),\mathcal{R}(E^{T,b})\right)
\right)=\cos(\beta_{E,F}).
\]
In this case we have the following result. See Figure \ref{fig2} for an illustration.
\begin{corollary}\label{corollary:betaf} Under the assumptions of Theorem \ref{theorem:nonsymetricgeneral} we have
\begin{equation}\label{betaEF}
\kappa(FE^{T,b}) =\|FE^{T,b}\|_{\color{black}a,a}\|(FE^{T,b})^{-1}\|_{a,a}
\leq   \|\widehat{F}\|_{\color{black}a,b}\|F\|_{\color{black}b,a}
\frac{\cos(\alpha_E)}{\cos(\beta_{E,F})} \|E\|_{\color{black}b,a}\|\widehat{E}\|_{\color{black}a,b}
\end{equation}
where $\alpha_E$ is defined in \eqref{reveq:dev:alphaE} and $\beta_{E,F}$ is defined in \eqref{eq:def:betaf}.
\end{corollary}

Then we can try to study the angle related to subspaces 
$\mathcal{R}(F^{T,b})$ and $\mathcal{R}(\widehat{F})$, 
 $\mathcal{R}(\widehat{F})$ and  $\mathcal{R}(\widehat{E})$ and 
 the angles between {\color{black}subspaces} $\mathcal{R}(\widehat{E})$ and 
  $\mathcal{R}(E^{T,b})$. 
  Recall that we have (\ref{angleEtilde}) and the analogous expression 
  for $F$, that is  
  \begin{equation}\label{angleEtilde2}
\cos(\Theta_b(\mathcal{R}(\widehat{F}) , \mathcal{R}(F^{T,b}))=\sin(\theta_b(\mathcal{R}(\widehat{F}) , \mathcal{N}(F)) = \frac{1}{\| Q_F \|_{\color{black}b,b}} \geq 
\frac{1}{\|\widehat{F}\|_{\color{black}a,b} \|F\|_{\color{black}b,a}}.
\end{equation}
Here $Q_F=\widehat{F}F$.

      In order to make the presentation simpler we only present the case where we can chose $\widehat{E}$ and $\widehat{F}$ such that 
$\mathcal{R}(\widehat{F})=\mathcal{R}(\widehat{E})$. Recall that 
$Q_E=\widehat{E}E$ and $Q_F=\widehat{F}F$.

\begin{theorem} Consider the assumptions of Theorem \ref{theorem:nonsymetricgeneral}.
Assume additionally $E\not = F$ and that  the following two conditions hold,
\begin{enumerate}
\item We 
can chose $\color{black}\widehat{E}$ and $\widehat{F}$ such that 
$\widehat{H}=\mathcal{R}(\widehat{F})=\mathcal{R}(\widehat{E})
\subset G$.  
\item 
It holds 
\[
\beta_E+
 \beta_F< \frac{\pi}{2},
\]
where $$\beta_E=\Theta_b(\widehat{H} , \mathcal{R}(E^{T,b})) = 
\cos^{-1}(\|Q_E\|^{-1}_{b,\color{black}b},
$$ and $$\beta_F=\Theta_b( \widehat{H} , \mathcal{R}(F^{T,b}))=
\cos^{-1}(\|Q_F\|^{-1}_{b,\color{black}b}.$$
\end{enumerate}
Then we have, 
$$\|(EF^{T,b})^{-1}\|_{\color{black}a,a}\leq 
\frac{\cos(\alpha_E)}{\cos(\beta_{E}+\beta_{F})}\|\widehat{E}\|_{\color{black}a,b}\|\widehat{F}\|_{\color{black}b,a}.$$
\end{theorem}

For other characterizations of $\|Q\|_{a,b}$ and  $\|Q|_{\mathcal{R}(\widehat{E})}\|_{a,b}$ that may lead to possible analysis of nonsymmetric method see \cite{Daniel, 2p, galantai2013projectors}.

\subsection{Special nonsymmetric methods}\label{specialanalysis}

We consider the case of the family of nonsymmetric methods of Section 
\ref{sec:nonsymmetric}, in particular we focus on Example 
\ref{especial}. For simplicity of the presentation we consider only the case  where $S=R$. 
The general case can be also consider from the results presented next. In the case $S=R$
 we  can estimate the norm    $\|Q|_{\mathcal{R}(\widehat{E})}\|_{\color{black}b,b}$ in a simple way.

\begin{figure}[ht]
    \centering
    \includegraphics[width=0.65\textwidth]{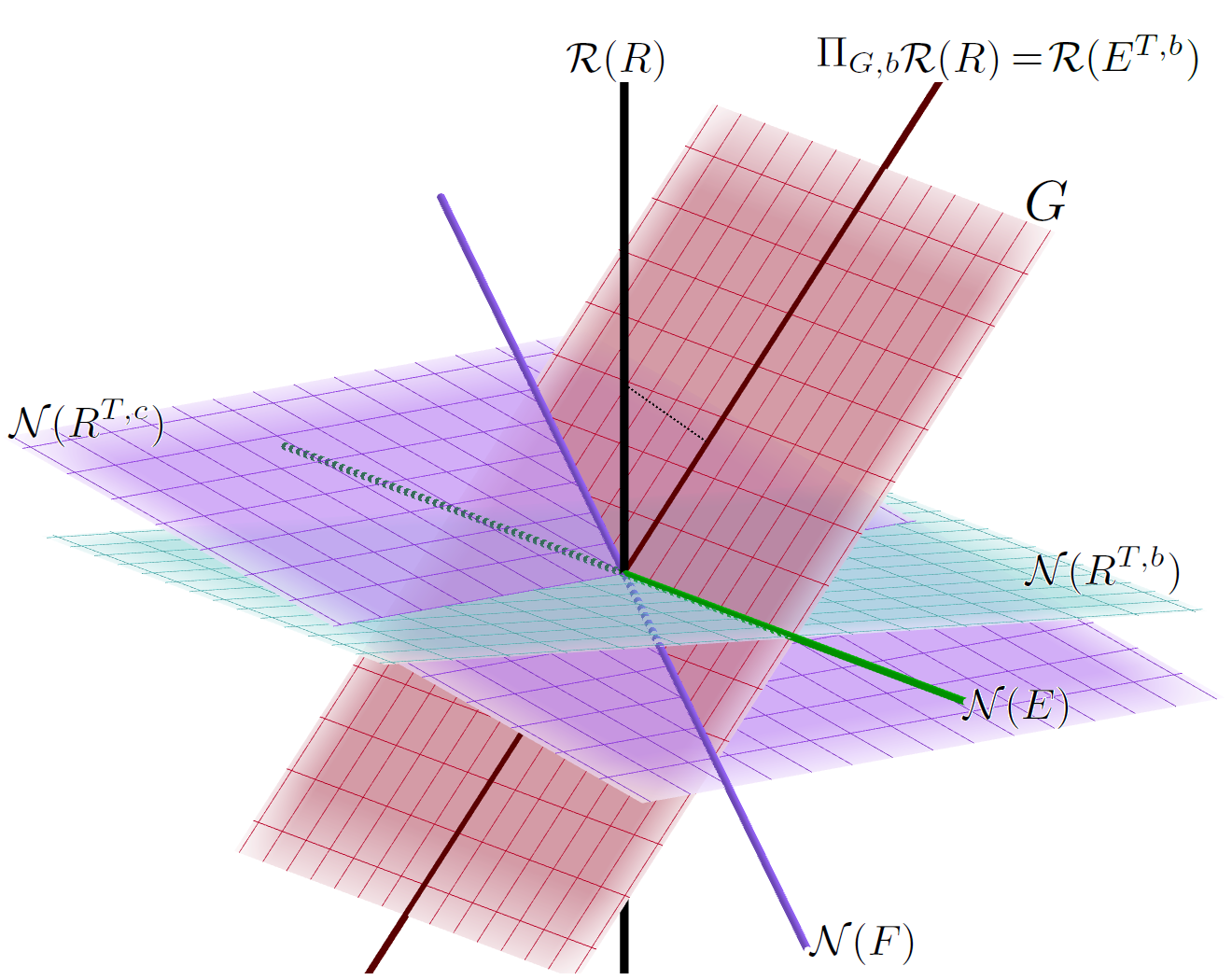}
    \caption{Illustration  of relative position of subspaces of ${\color{black}\widehat{G}}$. In particular, we illustrate the angle between subspaces 
    $\mathcal{N}(E)$ and $\mathcal{N}(F)$. }\label{fig3}
\end{figure}

%


\begin{theorem}\label{final} Assume there is a bounded restriction operator $R:H\to {\color{black}\widehat{G}}$
and bilinear forms $b$ and $c$ such that $E=R^{T,b}|_G$ and $F=R^{T,c}|_G$ where 
$G\subset {\color{black}\widehat{G}}$. Suppose that the extensions operators $E$ and $F$ have stable right 
inverse $\widehat{E}$ and $\widehat{F}$, respectively.
Assume also that 
\[
 c(\phi,\phi) \leq r^2_0 b(\phi,\phi) \mbox{ for all }\phi \in {\color{black}\widehat{G}}
\]
and 
\[b(\psi,\psi)\leq r_1^2 c(\psi,\psi) \mbox{ for all } 
\psi \in \mathcal{R}(R).
\]
We have 
\[
\|\Pi_{\mathcal{R}(R),c}\|_{\color{black}b,b} \leq r_0r_1
\]
and therefore, 
\begin{equation}
    \cos \left( 
\Theta_b
\left( \mathcal{N}(R^{T,b} ),  \mathcal{N}(R^{T,c}) \right)   \right) = 
\frac{1}
{
\|\Pi_{\mathcal{R}(R),c}\|_{\color{black}b,b} 
}
\geq \frac{1}{ r_0r_1 }. 
\end{equation}
Finally we can bound 
\begin{equation}\label{FETfinal}
\kappa( FE^{T,b} )= \kappa( R^{T,c} \Pi_{G,b} R  ) 
\leq   \|\widehat{F}\|_{\color{black}a,b}\|F\|_{\color{black}b,a}\cos(\alpha_E) r_0r_1 
\|E\|_{\color{black}b,a}\|\widehat{E}\|_{\color{black}a,b}.
\end{equation}

\end{theorem}

\begin{proof} 
 Note that in this case there {\color{black}are} bilinear forms 
 $b$ and $c$ such that $E=R^{T,b}|_G$ and $F=R^{T,c}|_G$. We then have, 
\begin{equation}\label{eq:bc0}
    \mathcal{N}(E)=\mathcal{N}(R^{T,b})\cap G 
    \quad \mbox{ and }\quad 
    \mathcal{N}(F)=\mathcal{N}(R^{T,c})\cap G. 
\end{equation}
Therefore, 
\begin{equation}\label{eq:bc1}
\beta_{E,F}=
\Theta_b\left( \mathcal{N}(E), \mathcal{N}(F)\right)\leq \Theta_b
\left( \mathcal{N}(R^{T,b} ),   \mathcal{N}(R^{T,c}) \right).
\end{equation}
It is clear that  no element in 
$\mathcal{N}(R^{T,b} )$ is $c-$orthogonal  to 
the space $\mathcal{N}(R^{T,c} )$.  Note that (by using \eqref{Thetatheta}),
\begin{equation}\label{eq:bc2}
\cos(
\Theta_b
\left( \mathcal{N}(R^{T,c}), \mathcal{N}(R^{T,b} )    \right)
)=\sin(\theta_b( \mathcal{N}(R^{T,c} ) , \mathcal{R}(R)).
\end{equation}
The $c$-orthogonal projection $\Pi_{\mathcal{R}(R),c}$ is the 
$b-$oblique projection on 
$ \mathcal{R}(R)$  and in the direction of 
$\mathcal{N}(R^{T,c} )$. That is, $\Pi_{\mathcal{R}(R),c}=
Q(\mathcal{R}(R),\mathcal{N}(R^{T,c} ))$. We only need to estimate the $c-$norm of 
this projection. See an illustration in Figure \ref{fig3}.

Due to Corollary \ref{corollary:betaf} we need only to bound 
$\cos(\beta_{E,F})$ defined in \eqref{eq:def:betaf}.  By \eqref{eq:bc0},
\eqref{eq:bc1}, \eqref{eq:bc2} and (\ref{angle})  we need only to bound 
the norm $\|\Pi_{\mathcal{R}(R),c}\|_{a,b}$. 

We have, 
\begin{eqnarray*}
\|\Pi_{\mathcal{R}(R),c}\|_{\color{black}b,b}&=&
\sup_{\phi\not =0} \frac{ \|\Pi_{\mathcal{R}(R),c} \phi\|_{\color{black}b}}{\|\phi\|_{b}}\\
&\leq & r_0r_1 
\sup_{\phi\not =0} \frac{ \|\Pi_{\mathcal{R}(R),c} \phi\|_c}{\|\phi\|_c}
\\&\leq & r_0r_1 .
\end{eqnarray*}
\end{proof}

 If no more information is available about the restriction operator $R$ we can use the  the following result.
\begin{lemma}
Under the assumption of Theorem \ref{final} we can bound
$r_1\leq \|R\|_{a,b} \|\widehat{F}\|_c .$
\end{lemma}
\begin{proof}
Note that for any $\psi=Rw\in \mathcal{R}(R)$ we can combine Corollaries 
\ref{corolaryR2} and \ref{corolaryRc} (with $S=R$) to obtain 
\begin{equation}
    \|\psi\|_{b} \leq \|R\|_{\color{black}a,b} \|w\|_a\leq \|R\|_{\color{black}a,b}\|\widehat{F}\|_{a,c}\|\psi\|_c .
\end{equation}
\end{proof}

In the case of the operator in \eqref{eq:FETbRTcPiGbR}, 
we note that if the image of $R$ is in the {\color{black}appropriate} relative position (in sense of angles measured in the $c-$inner product) with respect to the subspace $G$ and $G^{\perp,b}$, then  the operator 
\eqref{eq:FETbRTcPiGbR} is positive definite in the sense that 
$c( FE^{T,b}u,u)\geq 0$ for all $u\in H$. See Remark \ref{positivity}.

Note that 
$\Pi_{G,b}$ is the ($c$-oblique) projection onto 
$G$ and in the direction of
$G^{\perp,b}$. 
Let us consider $y\in G$,
$z\in G^{\perp,c}$ and put 
$x=y+z$. Any $x\in {\color{black}\widehat{G}}$ can be obtained 
in this manner.  Note that
$\Pi_{G,b}x=y+\Pi_{G,b}z$
\[
c(x,\Pi_{G,b}x)=
c(y,y)+c(y,\Pi_{G,b}z)
\geq 
\frac{1}{2}\left( 
\|y\|_c^2 -
\|y\|_c\|\Pi_{G,b}z\|_c \right).
\]
We conclude that if 
$x$ is such that 
\begin{equation}
    \|\Pi_{G,b}z\|_c^2 \leq 
    \|y\|_c^2
\end{equation}
then $c(x,\Pi_{G,b}x)\geq 0$. 
This happens in particular if 
\begin{equation}\label{eq:aux46}
    \|z\|_c^2 
    \|\Pi_{G,b}|_{ G^{\perp,c}}\|_c^2 \leq 
    \|y\|_c^2.
\end{equation}
We have the following result. 
\begin{theorem}\label{boundedtan}
Assume that ${\color{black}\widehat{G}}$ is finite dimensional. If 
$y\in G$,
$z\in G^{\perp,c}$ and $x=y+z$ with 
$\|z\|_c  \tan(\Theta_c\left( G^{\perp,b},
         G^{\perp,c}\right))\leq  \alpha\|y\|_c $ with 
         $\alpha\leq 1$, then 
${\color{black}c}(x,\Pi_{G,b}x)\geq (1-\alpha^2) \|y\|^2_c$.

\end{theorem}

\begin{proof}
According to \eqref{eq:aux46} we need only to bound 
$\|\Pi_{G,b}|_{ G^{\perp,c}}\|_c$. 
Recall that $\Pi_{G,b}$ is the $c-$oblique projection onto 
$G$ and in the direction of $ G^{\perp,b}$. Using \cite[Lemma 2.80 (p.76)]{galantai2013projectors} we can 
bound the norm  $\|\Pi_{G,b}|_{ G^{\perp,c}}\|_c$ as follows, 
\begin{eqnarray}
        \|\Pi_{G,b}|_{ G^{\perp,c}}\|_{\color{black}c,c} \leq 
         \|\Pi_{G,b}\|_{\color{black}c,c}\sin\left( \Theta_c\left( G^{\perp,b},
         G^{\perp,c}\right)\right)\\
         =\frac{\sin\left( \Theta_c\left( G^{\perp,b},
         G^{\perp,c}\right)\right)}{\sin(\theta_c(G, G^{\perp,b} ))}\\
     =\frac{\sin\left( \Theta_c\left( G^{\perp,b},
         G^{\perp,c}\right)\right)
     }{\cos(\Theta_c\left( G^{\perp,b},
         G^{\perp,c}\right))} \\
                  =\tan(\Theta_c\left( G^{\perp,b},
         G^{\perp,c}\right)).
\end{eqnarray}
Here we have used \eqref{angle}.
\end{proof}

Introduce the operator $\color{black}C:{\color{black}\widehat{G}}\to {\color{black}\widehat{G}}$ defined by 
\begin{equation}\label{eq:defC}
    b(Cu,v)=c(u,v).
\end{equation}
This operator is symmetric and positive definite. We recall the  Wielandt inequality. See for instance \cite{MR2997829,galantai2013projectors}.
\begin{lemma}Assume that  $m b(x,x)\leq c(x,x)\leq Mb(x,x)$ for all 
$x\in {\color{black}\widehat{G}}$. For any pair of vectors $z,w$ with $b(z,w)=0$ we have 
\begin{equation}
    \frac{b(z,Cw)^2}{b(z,Cz)b(w,Cw)} \leq \left(\frac{M-m}{M+m}\right)^2
\end{equation}
\end{lemma}
Taking $x=C^{-1/2}z$ { \color{black} and} $y=C^{1/2}w$ we obtain, 

\begin{corollary}\label{corwielandt}
For any pair of vectors $x,y$ with $b(x,y)=0$ we have 
\begin{equation}
    \frac{b(x,Cy)^2}{b(x,x)b(Cy,Cy)} \leq \left(\frac{M-m}{M+m}\right)^2
\end{equation}
\end{corollary}

We see that $\cos(\theta_c(G, G^{\perp,b} )) \leq \left(\frac{M-m}{M+m}\right)$. Then \\
$\sin(\theta_c(G, G^{\perp,b} )) \geq 
\sqrt{1-\left(\frac{M-m}{M+m}\right)^2} =
\frac{\sqrt{mM}}{\color{black}2(M+m)}$  and therefore\\
$\tan(\Theta_c\left( G^{\perp,b},
         G^{\perp,c}\right))=1/\tan(\theta_c(G, G^{\perp,b} )) \leq \frac{\sqrt{mM}(M-m)^2}{\color{black}2(M+m)}$.
We have the following corollary of the previous result and Theorem \ref{boundedtan}.
\begin{theorem}[Positivity of special non-symmetric methods]\label{positivity} Assume that there is a constant $\alpha_R$ such that
\begin{equation}
 \| (I-\Pi_{G,c})Ru \|_{\color{black}c,c}
 \frac{\sqrt{mM}(M-m)^2}{\color{black}2(M+m)} \leq \alpha_R\| \Pi_{G,c}R u\|_{\color{black}c,c}
\quad \mbox{ for all } u\in H.
\end{equation}
Then we have that $c(R^{T,c} \Pi_{G,b} R u,u)\geq 
(1-\alpha_R^2)c( \Pi_{G,c} R u,  \Pi_{G,c} R u)$.
\end{theorem}

\begin{remark}
Similar results hold when 
$S\not=R$ and $c=b$. In this case $E=R^{T,b}|_G$ and $F=S^{T,b}|_G$, 
\begin{equation}\label{eq:bc0add}
    \mathcal{N}(E)=\mathcal{N}(R^{T,b})\cap G 
    \quad \mbox{ and }\quad 
    \mathcal{N}(F)=\mathcal{N}(S^{T,b})\cap G 
\end{equation}
Therefore, 
\begin{equation}\label{eq:bc1add}
\beta_{E,F}=
\Theta_b\left( \mathcal{N}(E), \mathcal{N}(F)\right)\leq \Theta_b
\left( \mathcal{N}(R^{T,b} ),   \mathcal{N}(S^{T,b}) \right).
\end{equation}
This last angle can be bound in terms of $\Theta_b( \mathcal{R}(S) , \mathcal{R}(R))$, that would require and assumption on $S$ 
{\color{black}when} compared to $R$.
\end{remark}
\section{Restricted methods}\label{sec:restricted} 
In this section we consider a particular case of Example \ref{original}.
We now use the Hilbert space framework previously introduced to obtain a bound for {\color{black}the condition number of a} restricted additive method. For simplicity and readability we consider the one level method. Similar results can be obtained using a multilevel setting. 
We use the notation and setup introduced in Section
\ref{sec:classical}, in particular, we consider the Hilbert spaces   $H$ and  $G=\times_{\ell=1}^{N_S} H^1_0(\mathcal{O}_\ell)$  as before (with the same inner products $a$ and $b$). \\

Start by defining the (harmonic-like) extension operator $F_\ell : H^1_0(\mathcal{O}_\ell)\to H$ as follows. 
Given $v_\ell\in H^1_0(\mathcal{O}_\ell)$ define $F_\ell v_\ell\in H$ as the unique solution of 
\begin{equation}\label{eq:def:Fell}
\int_D \nabla F_\ell v_\ell \nabla z = \int_{D_\ell} \nabla v_\ell \nabla z \quad \mbox{ for all  } z \in H.
\end{equation}
Note that the integration on the right is on the domain $D_\ell$. In the case of an interior subdomain, this is the weak form of the strong form given by, 
\begin{equation}
 \left\{\begin{array}{rll}
-\Delta F_\ell v_\ell&={\color{black}-}\Delta v_\ell & \mbox{ in } D_\ell,\\
-\Delta F_\ell v_\ell&=0 & \mbox{ in } D\setminus D_\ell,\\
\frac{\partial F_\ell v_\ell}{\partial \eta^+}+
\frac{\partial F_\ell v_\ell}{\partial \eta^-}
&=   
\frac{\partial v_\ell}{\partial \eta}  & \mbox{ on } \partial D_\ell,\\
F_\ell v_\ell&=0 & \mbox{ on } \partial D.\\
\end{array}\right.
\end{equation}
We introduce the bilinear from $\widetilde{b}_\ell$ defined by 
\[
\widetilde{b}_\ell(v, z)=\int_{D_\ell} \nabla v \nabla z
\]
for any $v$ and $z$ that can be restricted to $D_\ell$. 
Using our bilinear forms notation we have 
\[
a(F_\ell v_\ell , z) = \widetilde{b}_\ell(v_\ell, z) \quad \mbox{ for all } z\in H.
\]
Define, in analogy with the previous discussions, the extension operator $F: G\to H$ by
\begin{equation}\label{review:def:F}
F \{ v_\ell\}= \sum_{\ell=1}^{N_S} F_\ell v_\ell.
\end{equation}
Consider also the operator $ F^{T,b}_\ell$  which are given by the problem, 
$F^{T,b}_\ell u\in H_0 ^1(\mathcal{O}_\ell)$ with, 
\begin{equation}\label{reveqforneumannmatrix}
\int_{\mathcal{O}_\ell}\nabla F^{T,b}_\ell u\nabla v_\ell =
\int_{D} \nabla   u \nabla F_\ell v_\ell =\int_{D_\ell} \nabla u \nabla v_\ell \quad \mbox{ for all } v_\ell \in H^1_0(\mathcal{O}_\ell).
\end{equation}
Here in the last step we used the definition of $F_\ell$. Note that the weak from 
above correspond to the strong from
\begin{equation}\label{ho}
 \left\{\begin{array}{rll}
-\Delta F_\ell ^{T,b} v_\ell&=-\Delta u & \mbox{ in } D_\ell,\\
-\Delta F_\ell^{T, {\color{black}b}} v_\ell&=0 & \mbox{ in } \mathcal{O}_\ell\setminus D_\ell,\\
\displaystyle \frac{\partial F_\ell^{T, {\color{black}b}}  v_\ell}{\partial \eta^+}+
\frac{\partial F_\ell^{T,b} v_\ell}{\partial \eta^-}
&=   \displaystyle 
\frac{\partial u}{\partial \eta^+}  & \mbox{ on } \partial D_\ell,\\
F_\ell^{T, {\color{black}b}}  v_\ell&=0 & \mbox{ on } \partial  \mathcal{O}_\ell.\\
\end{array}\right.
\end{equation}
{\color{black}This equation corresponds to a local problem with the same computational cost of the local problem used to
obtain $E_\ell^{T,b}$ in  the additive method}. We then have 
\[
F^{T,b}= \{ F_\ell^{T,b}\}.
\]

\subsection{\color{black} The operator $EF^{T,b}$}
Let  $u$ be the solution of \eqref{problem} and introduce 
$w$ such that $EF^{T,b}w=u$. Then we consider the equation,  
$$
a(EF^{T,b}w,v)=L(v) \mbox{ for all }  v\in H.
$$
Note that, given $w$ the computation of $EF^{T,b}w=\sum_{\ell=1}^{N_S}
E_\ell F_\ell^{T,b}w$ requires the solution of local problems posed on the overlapping subdomais. See \eqref{reveqforneumannmatrix} and 
\eqref{ho}. Then we can iteratively solve this equation. After computing $w$ we can compute $u=EF^{T,b}w$ by solving one more round of local problems. 

{\color{black}
\begin{remark}\label{rev:remarkclarify}
 For comparison the 
local problem of the OBDD in \cite{kimn2007obdd} can be written as
\begin{equation}\label{reveq:strong:cutressolve}
-\Delta w_\ell= {\eta}^{cut}_\ell \Delta v_\ell \mbox{ in  } \mathcal{O}_\ell
\end{equation}
where ${\eta}^{cut}_\ell $ is a  cut function that is $1$ in $D_\ell\subset
\mathcal{O}_\ell$ and decays to zero. See Example \ref{original}.  See the comments in Example \ref{original}. The method analyzed in this section use local problems in \eqref{ho}. A main difference is that (when applied to e.g. finite element implementations) the local
problem \eqref{ho} needs the Neumann stiffness matrix associated to subdomains $\{D_\ell\}_{\ell=1}^{N_S}$ which is not the case 
for the local problem in \eqref{reveq:strong:cutressolve}. See the weak form of \eqref{ho}  in 
\eqref{reveqforneumannmatrix}.
\end{remark}}


\subsection{\color{black} The operator $FE^{T,b}$}
If now we consider the method $FE^{T,b}$ which corresponde to the RAS preconditioner. The solution of problem 
\eqref{problem} satisfies,
\begin{eqnarray*}
a(FE^{T,b}u,v)&=&a(E^{T,b}u,F^{T,b}v) \\
&=&\sum_{\ell=1}^{N_S} \int_{\mathcal{O}_\ell}\nabla E^T_\ell u \nabla F_\ell^{T,b} v\\
&=&\sum_{\ell=1}^{N_S} \int_{D_\ell}\nabla E^{T,b}_\ell u \nabla v\\
&=&L(EF^{T,b}v)= \widetilde{L}(v).
\end{eqnarray*}
Here we see that each term $\int_{D_\ell}\nabla E^{T,b}_\ell u \nabla v$ 
can be computed to assemble $\widetilde{L}$.

After $\widetilde{L}$ is assembled, we solve iteratively
$$
\sum_{\ell=1}^{N_S} \int_{D_\ell}\nabla E^T_\ell u \nabla v  = 
\widetilde{L}(v).
$$
Recall that the right hand above is equivalent 
to $a(FE^{T,b}u,v)$. Note that the computation of the residual needs to update 
the solution only on the subdomains $D_\ell$.  {\color{black}Note also that in the case of an implementation in finite element spaces we need access to the 
Neumann stiffness matrix associated to subdomains $\{D_\ell\}_{\ell=1}^{N_S}$}.

\section{\color{black} Condition number estimates }\label{sec:condfor-ras}

We can consider the operator $F$ introduced in Section 
\ref{sec:restricted} and use the results of our Hilbert space framework to obtain the non-singularity of $EF^{T,b}$ or $FE^{T,b}$ as before. For $0<\epsilon$ define
\[
c_\epsilon(\{u_\ell\},\{v_\ell\})=
\sum_{\ell=1}^{N_S}\int_{D_\ell} \nabla u_\ell \nabla 
v_\ell+
\epsilon\sum_{\ell=1}^{N_S}\int_{\mathcal{O}_\ell\setminus D_\ell} \nabla u_\ell \nabla 
v_\ell+
\sum_{\ell=1}^{N_S}b_\ell^{\partial}(u_\ell,v_\ell) .
\]

Recall the restriction operator $R$ introduced in Section
 \ref{sec:classical}. Define $F_\epsilon= R^{T,c_\epsilon}|_G$ so that for $\{v_\ell\}\in G$ we have 
\[
\int_D \nabla F_\epsilon^{T,c_\epsilon }\{v_\ell\}\nabla z=
\sum_{\ell=1}^{N_S}\int_{D_\ell} \nabla 
v_\ell\nabla z+
\epsilon\sum_{\ell=1}^{N_S}\int_{\mathcal{O}_\ell\setminus D_\ell} \nabla 
v_\ell \nabla z
  \quad \mbox{ for all }
z\in H.
\]
Denote by  $F=F_0$ is the extension operator used in Section \ref{sec:restricted}. Define the operator $F_{ov}$ by 
\begin{equation}\label{eq:def:opT}
\int_D \nabla F_{ov}\{v_\ell\}\nabla z=\sum_{\ell=1}^{N_S}\int_{\mathcal{O}_\ell\setminus D_\ell} \nabla 
v_\ell \nabla z
  \quad \mbox{ for all }
z\in H.
\end{equation}
The operator $F_{ov}$ is clearly bounded with $\|F_{ov}\|_{a,b}\leq \|E\|_{a,b}$
and $$F_\epsilon={\color{black}F}+\epsilon F_{ov}$$
{\color{black}where $F$ was defined in \eqref{review:def:F}.}

Note that 
$
F_\epsilon E^{T,b}= R^{T,c_\epsilon} \Pi_{G,b} R 
$
and therefore we are in the case of special non-symmetric methods 
of Section \ref{sec:nonsymmetric} that were analyzed in 
Section \ref{specialanalysis}.\\

We can find {\color{black}a} stable right inverse of $F_\epsilon$ as follows.
\begin{lemma}[Stable right inverse of $F$]\label{hatF} There exits $C_{0F}$ such that 
for every $v\in H$ there exits $v_\ell\in H_0^1(\mathcal{O}_\ell)$ such that
\[
\sum_{\ell=1}^{N_S} F_\ell v_\ell = v
\]
and 
\[
\sum_{\ell=1}^{N_S} b_\ell(v_\ell,v_\ell) \leq C_F^2 a(v,v).
\]
If we put $\widehat{F}v=\{v_\ell\}$ we then have $\| \widehat{F}\|_{a,b}\leq C_F$. We can estimate
$C_F^2\preceq \nu C_{cut}(1+\frac{1}{\tau\delta})$. 
We also have 
 $\|\widehat{F}\|_{a,c_\epsilon}\leq C_F$.
\end{lemma}
\begin{proof}
This proof is similar to the stable decomposition for the operator $E$; see \cite{MR2104179}. Let us consider cut of functions $\eta_\ell$ 
introduced in \eqref{eq:def:cutoff}
Define $v_\ell=\eta_\ell v$. We have that 
\begin{eqnarray*}
a(\sum_{\ell=1}^{N_S}  F_\ell v_\ell, z)&=& 
\sum_{\ell=1}^{N_S}  a( F_\ell v_\ell,z)\\&=&
\sum_{\ell=1}^{N_S} \int_{D_\ell} \nabla v_\ell \nabla z\\
&=&\sum_{\ell=1}^{N_S} \int_{D_\ell} \nabla v \nabla z
= \int_{D} \nabla v \nabla z=a(v,z).
\end{eqnarray*}
We conclude that $\sum_{\ell=1}^{N_S}  F_\ell v_\ell =v$. As in the case of 
classical additive method stable decomposition -which uses the gradient of the product rule plus a Friedrichs inequality, it is easy to see that,
\[
\sum_{\ell=1}^{N_S} b_\ell(v_\ell,v_\ell) =
\sum_{\ell=1}^{N_S} \int_{\mathcal{O}_\ell}
|\nabla v_\ell|^2=
\sum_{\ell=1}^{N_S} \int_{\mathcal{O}_\ell}
|\nabla \eta_\ell v|^2 \preceq \nu C_{cut}(1+\frac{1}{\tau\delta}) \int_D |\nabla v|^2.
\]
\end{proof}
{\color{black}A} stable right inverse of $F_\epsilon$ can be also obtained.
\begin{corollary}\label{lemma:normFeps} For $\epsilon$ small enough, $F_\epsilon \widehat{F}$ is non-singular. Moreover, 
$\widehat{F}_\epsilon= \widehat{F}(F_\epsilon \widehat{F})^{-1}$ is an stable right inverse of $F_\epsilon$ with 
$\| \widehat{F}_\epsilon \|_{\color{black}b,b}\leq \|\widehat{F}\|_{\color{black}b,b}/(1-\epsilon \|F_{ov}\widehat{F}\|_{\color{black}a,a})$.
\end{corollary}
\begin{proof}
Note that
$F_\epsilon\widehat{F} = F\widehat{F} + \epsilon F_{ov}\widehat{F} = 
I+\epsilon  F_{ov}\widehat{F}$. This is invertible for small enough $\epsilon$ and 
$\| F_\epsilon\widehat{F}\|_{\color{black}b,b} \leq 1/(1-\epsilon \|F_{ov}\widehat{F}\|_{\color{black}b,b})$.
\end{proof}

We now estimate the norm of $F_\epsilon$.
\begin{lemma}[Norm $F$]\label{lemma:normF} We have $\|F\|_{\color{black}b,b}\leq 1$ and  $\|F_\epsilon\|_{\color{black}b,b}\leq 1+\epsilon \|F_{ov}\|_{\color{black}b,b}$.
\end{lemma}
\begin{proof}
From the definition of $F_\ell$ we have for every $z\in H$,
\begin{eqnarray*}
a(F \{v_\ell\}, z)=a(\sum_{\ell=1}^{N_S}  F_\ell v_\ell, z)&=& 
\sum_{\ell=1}^{N_S} \int_{D_\ell} \nabla v_\ell \nabla z\\
&=&\sum_{\ell=1}^{N_S} \int_{D_\ell} \nabla v \nabla z\\
&\leq & 
\left( \sum_{\ell=1}^{N_S} \int_{D_\ell} |\nabla v_\ell|^2 \right)^{1/2}
\left( \sum_{\ell=1}^{N_S} \int_{D_\ell} |\nabla z|^2 \right)^{1/2}\\
&\leq & 
\left( \sum_{\ell=1}^{N_S} \int_{\mathcal{O}_\ell} |\nabla v_\ell|^2 \right)^{1/2}
\left(  \int_{D} |\nabla z|^2 \right)^{1/2}\\
&\leq & \| \{v_\ell\}\|_{b} \|z\|_a.
\end{eqnarray*}
Taking $z=F\{v_\ell\}=\sum_{\ell=1}^{N_S}  F_\ell v_\ell$ we see that
$\| F\{v_\ell\}\|_a\leq \| \{v_\ell\}\|_{b}$ and the result follows.
\end{proof}
As a corollary we have the following result.
\begin{corollary}[Angle $\alpha_F$] \label{lemma:angle} We have $\alpha_F=\theta_b\left(\mathcal{R}(\widehat{F}),\mathcal{R}(F^{T,b})\right)=0$.
\end{corollary}
We do not need the following results but we 
stated for completeness. The range of $\widehat{E}$ and $\widehat{F}$  coincide. 
\begin{lemma}
We can chose $\widehat{E}$ and $\widehat{F}$ such that 
$\mathcal{R}(\widehat{E})=\mathcal{R}(\widehat{\color{black}F})$.
\end{lemma}
\begin{proof}
As in the proof of Lemma \ref{hatF} chose $\widehat{F}u=\{ \eta_\ell u\}.$
Define $$\eta=\sum_{\ell=1}^{N_S}{\color{black}\eta}_\ell.$$ We recall that an classical construction of the operator $\widehat{E}$ is given by 
$$
\widehat{E} u =  \{\chi_\ell u \} \mbox{ with }  \chi_\ell =  \frac{\eta_\ell}{\eta}.
$$
We readily see that $ \widehat{F} u = \eta \widehat{E} u$  and since $\eta\geq 1$ is bounded with bounded gradient we have the result.
\end{proof}
We can estimate the parameter $r_1$ in Theorem \ref{final} as follows.
\begin{lemma}\label{lemma:r1}
We have that 
$b(Rv,Rv)\leq \nu c_\epsilon(Rv,Rv)$ for all $v\in H$. 
Here $\nu$ is defined in \eqref{eq:def:nu}.
\end{lemma}
\begin{proof} Observe that, 
\begin{eqnarray*}
\sum_{\ell=1}^{N_S} b_\ell(v|_{\mathcal{O}_\ell},v|_{\mathcal{O}_\ell}) &=&
\sum_{\ell=1}^{N_S} \int_{\mathcal{O}_\ell}|\nabla v|^2\\
&\leq & \nu \int_D |\nabla v|^2\\
&=&\nu 
\sum_{\ell=1}^{N_S}\int_{D_\ell} |\nabla 
v_\ell|^2\\
&\leq& \nu \left( 
\sum_{\ell=1}^{N_S}\int_{D_\ell} |\nabla 
v_\ell|^2+
\epsilon\sum_{\ell=1}^{N_S}\int_{\mathcal{O}_\ell\setminus D_\ell} |\nabla 
v_\ell|^2\right).
\end{eqnarray*}
Therefore, if we include the boundary terms we have 
$b(Rv,Rv)\leq \nu c_\epsilon(Rv,Rv)$.
\end{proof}

Putting together the previous bounds and Theorem 
\ref{final} we can write condition number bounds.
Recall that:
\begin{itemize}
    \item For $E$ define in Section \ref{sec:classical} we have 
    $\|E\|_{\color{black} b,a}\leq \sqrt{\rho(\mu)}\leq \nu$. See Remark \ref{remarkclassicalSD}.
    \item For $E$ define in Section \ref{sec:classical} we have $\|\widehat{E}\|_{\color{black} a,b}\leq  C_E= \nu C_{pu}(1+1/(\tau\delta))$.
    \item  For $\hat{F}$ defined in Lemma \ref{hatF} we have 
    $\|\widehat{F}\|_{\color{black} a,b} \leq  C_F=\nu C_{cut}(1+1/(\tau\delta))$.
    \item From Lemma \ref{lemma:normF} we have  $\|F\|_{\color{black}b,a}\leq 1$, $\|F_\epsilon\|_{\color{black}b,a}\leq 1+\epsilon \|F_{ov}\|_{\color{black}b,a}$.
    \item From Corollary \ref{lemma:normFeps} we have 
    $\| \widehat{F}_\epsilon \|_{\color{black}b,a}\leq \|\widehat{F}\|_{\color{black}b,a}/(1-\epsilon \|F_{ov}\widehat{F}\|_{\color{black}a,a})$.
    \item  From Corollary \ref{lemma:angle} we have $\cos(\alpha_E)=1$.
    \item {\color{black}From} $c_\epsilon$, $0< \epsilon<1$,  defined in and $b$ defined in  Section \ref{sec:classical} we have
    $r_0 =1$.
    \item From Lemma \ref{lemma:r1} we have $r_1=\sqrt{\nu}$.
\end{itemize}
Replacing in \eqref{FETfinal} we get the following result.

\begin{theorem} Let  $E$ and $F_\epsilon$  be defined as before. 
Then we have that $F_\epsilon E^{T,b}$ is invertible and 
$$\kappa(F_\epsilon E^{T,b})
\leq 
\frac{ \sqrt{\nu}  {\color{black}\sqrt{\rho(\mu)}}
C_FC_E
(1+\epsilon \|F_{ov}\|_{\color{black}b,a})
\cos(\alpha_E)
}{(1-\epsilon \|F_{ov}\widehat{F}\|_{\color{black}a,a})} 
\leq 
\frac{ \sqrt{\nu} 
{\color{black}\sqrt{\rho(\mu)}}C_FC_E
(1+\epsilon \|F_{ov}\|_{\color{black}b,a})
}{(1-\epsilon \|F_{ov}\widehat{F}\|_{\color{black}a,a})}.
$$
\end{theorem}
Finally, taking $\epsilon\to 0$ we obtain the 
condition number of the RAS, 
$$\kappa(FE^{T,b}) \leq   \sqrt{\nu} 
\sqrt{\rho(\mu)}C_FC_E .
$$
Note that  we use 
{\color{black}${\sqrt{\rho(\mu)}}\leq \sqrt{\nu}
$} and $C_F\leq C_E$ we have that the bound of 
$\kappa(FE^{T,b})$ is smaller than 
the bound for $\kappa(EE^{T,b}).$ That is, the bound for the condition number of this restricted method is smaller than the bound obtained for AS.

We now investigate the positivity of the 
operator $F_\epsilon E^{T,b}$.
Recalling \eqref{eq:defC} we have 
$
c_\epsilon(B_\epsilon x,y)=b(x,y).
$
We also have that  $\epsilon b(x,x)\leq c_\epsilon(x,x)\leq b(x,x)$. 
In order to use  Theorem \ref{positivity} we need 
\begin{equation}
 \| (I-\Pi_{G,c_\epsilon})Ru \|_{\color{black}c_\epsilon,c_\epsilon} 
 \frac{\sqrt{\epsilon }(1-\epsilon)^2}{\color{black}2(1+\epsilon)} \leq \alpha_R\| \Pi_{G,c_\epsilon}R u\|_{\color{black}c_\epsilon,c_\epsilon}
\quad \mbox{ for all } u\in H.
\end{equation}
The left hand side multiplier of the norm vanishes when 
$\epsilon\to 0$ and therefore we conclude (using Theorem \ref{positivity}) that, for $\epsilon$ small enough, $c_\epsilon(FE^{T,b}u,u)\geq 0$ for all 
$u\in H$.

\section*{Acknowledgments}
The author  wants to thank  the discussions on restricted domain decomposition methods and related topics with  M. Sarkis and M. Dryja. 
Juan Galvis thanks partial support  from the European Union's Horizon 2020 research and innovation programme under
the Marie Sklodowska-Curie grant agreement No 777778 (MATHROCKS).

\bibliographystyle{plain}
\bibliography{references}

\end{document}